\documentclass[12pt]{amsart}

\usepackage{amssymb,latexsym}

\usepackage{enumerate}
\usepackage{comment}

\makeatletter

\@namedef{subjclassname@2010}{

  \textup{2010} Mathematics Subject Classification}

\usepackage{amssymb, amsmath,amsthm}
\newtheorem{thm}{Theorem}[section]
\newtheorem{prop}[thm]{Proposition}

\newtheorem{cor}[thm]{Corollary}
\newtheorem{lem}[thm]{Lemma}
\theoremstyle{definition}
\newtheorem{rem}[thm]{Remark}

\newcommand{\ra}{\rightarrow}
\newcommand{\bk}{\backslash}
\newcommand{\mc}{\mathcal}

\newcommand{\mb}{\mathbb}
\newcommand{\sg}{\sigma}

\newcommand{\llf}{\left\lfloor}

\newcommand{\e}{\varepsilon}
\newcommand{\rrf}{\right\rfloor}

\renewcommand{\bar}{\overline}

\begin{document}
\title{Short Character Sums and the P\'{o}lya-Vinogradov Inequality}
\author{Alexander P. Mangerel}
\address{Centre de Recherches Math\'{e}matiques, Universit\'{e} de Montr\'{e}al, Montr\'{e}al, Qu\'{e}bec}
\email{smangerel@gmail.com}
\begin{abstract}
We show in a quantitative way that any odd primitive character $\chi$ modulo $q$ of fixed order $g \geq 2$ satisfies the property that if the P\'{o}lya-Vinogradov inequality for $\chi$ can be improved to 
$$\max_{1 \leq t \leq q} \left|\sum_{n \leq t} \chi(n)\right| = o_{q \ra \infty}(\sqrt{q}\log q)$$ then for any $\e > 0$ one may exhibit cancellation in partial sums of $\chi$ on the interval $[1,t]$ whenever $t > q^{\e}$, i.e.,$$\sum_{n \leq t} \chi(n) = o_{q \ra \infty}(t) \text{ for all $t > q^{\e}$.}$$
We also prove a converse implication, to the effect that if all odd primitive characters of fixed order dividing $g$ exhibit cancellation in short sums then the P\'{o}lya-Vinogradov inequality can be improved for all odd primitive characters of order $g$. Some applications are also discussed.
\end{abstract}
\maketitle
\section{Introduction and Main Results}
A central problem in analytic number theory concerns sharply estimating the partial sums of non-principal Dirichlet characters. There are two important types of problems regarding character sums to which a wealth of literature is devoted. \\
Let $\chi$ be a non-principal Dirichlet character modulo $q$, with $q$ large. A first, typically harder, problem concerns demonstrating estimates for short sums in the form
\begin{equation}\label{Cancel}
\sum_{n \leq t} \chi(n) = o_{q \ra \infty}(t)
\end{equation}
for any $t > q^{\e}$ and any $\e > 0$. Except under exceptional circumstances (e.g., when $q$ is smooth, see \cite{Iwa} and \cite{GrRi}), such estimates have been difficult to demonstrate. 
The best general result in this direction, due to Burgess, shows that such cancellation occurs (at least for $q$ cube-free) provided that $t > q^{1/4+\e}$. For this reason, we shall refer to estimates like \eqref{Cancel} as being of \emph{Burgess-type}.\\
A second type of problem concerns estimating the maximal size of partial sums of Dirichlet characters. A classical result in this direction, proven independently by P\'{o}lya and I.M. Vinogradov states that 
\begin{equation}\label{ClassPV}
\max_{1 \leq t \leq q} \left|\sum_{n \leq t} \chi(n)\right| \ll \sqrt{q}\log q.
\end{equation}
This result has subsequently been improved assuming the Generalized Riemann Hypothesis to the best possible upper bound\footnote{For $k \geq 1$, we write $\log_k q = \log_{k-1}(\log q)$, and $\log_1 q = \log q$.} $O(\sqrt{q}\log_2 q)$ by Montgomery and Vaughan, and this is best possible according to an old construction due to Paley (\cite{MV}, \emph{in loc. cit.}). However, unconditional improvements to \eqref{ClassPV} have been hard to come by in general (though see Remark \ref{RemOddOrd} below). We refer to the problem of improving \eqref{ClassPV} as one of \emph{PV-type}. \\
A priori, these two problems are not directly connected. Nevertheless, 
the work of Bober and Goldmakher \cite{BoGo} on one hand, and the more recent paper of Fromm and Goldmakher \cite{FrGo} have revealed a connection between hypothetical improvements to the P\'{o}lya-Vinogradov inequality and estimates for short quadratic character sums. For instance, Bober and Goldmakher \cite{BoGo} showed that if the P\'{o}lya-Vinogradov inequality can be improved for any \emph{even}\footnote{As usual, we call a character $\xi$ \emph{even} if $\xi(-1) = 1$, and \emph{odd} if $\xi(-1) = -1$.} character $\xi$ modulo $q$ in the form
\begin{equation}\label{PVIMP}
\max_{1 \leq t \leq q} \left|\sum_{n \leq t} \xi(n)\right| \leq \sqrt{q}\frac{\log q}{f(q)}
\end{equation}
for some non-decreasing function $f: \mb{R} \ra \mb{R}$ tending to infinity, then for any \emph{odd} character $\chi$ modulo $m$ there is an integer $n \ll_{\e} m^{\e}$ such that $\chi(n) \neq 0,1$. This result was extended (still in the case of quadratic characters $\xi$ and $\chi$) in \cite{FrGo}, where it was shown that a hypothesis like \eqref{PVIMP} for quadratic \emph{even} characters implies the bound $\sum_{n \leq t} \chi(n) = o(t)$ for any $t > m^{\e}$, whenever $\chi$ is a quadratic \emph{odd} character of prime conductor. \\
In this paper, we extend the latter results in two ways. First, we show that a hypothesis like \eqref{PVIMP} for an odd character $\chi$ implies a Burgess-type bound for $\chi$ itself in intervals of length $q^{\e}$. Second, we extend this result to characters of \emph{any} fixed order, rather than just for quadratic characters. Moreover, our result quantifies the range $t > q^{\e}$ as well as the $o(t)$ bound here.\\
\begin{thm}\label{THMGEN}
Let $g \geq 2$ and let $q$ be large. Let $a(t)$ be a non-decreasing function such that\footnote{A classical result of Schur implies that $\max_{1 \leq N \leq q}|\sum_{n \leq N} \chi(n)| \gg \sqrt{q}$, and hence $a(q) \ll \log q$ for any non-principal character $\chi$. We will thus have $\log a(q) \ll \log_2 q$ in general, and the above assumption is rather mild. For instance, it would easily hold for any function of the form $a(t) = (\log_k t)^A$, for any $k \geq 1$ and $A > 0$.} $\log a(q)\asymp \log a(t)$ for all $t > \exp\left(\sqrt{\log q}\right)$. Let $\chi$ be a primitive odd character modulo $q$ with order $g$, and assume that the P\'{o}lya-Vinogradov inequality for $\chi$ can be improved in the form
$$\max_{1 \leq t \leq q} \left|\sum_{n \leq t} \chi(n) \right| \leq \sqrt{q} \frac{\log q}{a(q)}.$$
Then for any $t > \exp\left(\frac{\log q}{a(q)^{1/6}}\right)$, we have
$$\sum_{n \leq t} \chi(n) \ll_g t\left(\frac{\log_2 a(t)}{\log a(t)}\right)^{1/(19g^2)}.$$
\end{thm}
As a consequence of Theorem \ref{THMGEN}, we have the following corollary regarding the relationship between improvements to the P\'{o}lya-Vinogradov inequality and Burgess-type bounds for fixed order characters.
\begin{cor} \label{Qual}
Let $g \geq 2$, and let $\chi$ be an odd primitive character with modulus $q$ and order $g$. Assume that the P\'{o}lya-Vinogradov inequality can be improved in the form
\begin{equation}
\max_{1 \leq t \leq q} \left|\sum_{n \leq t} \chi(n)\right| = o_{q \ra \infty}(\sqrt{q}\log q).
\end{equation}
Then for any $\e > 0$ and any $t > q^{\e}$ we have
$$\sum_{n \leq t} \chi(n) = o_{q \ra \infty}(t).$$
\end{cor}
The above results demonstrate that in order to prove cancellation in short sums of a given odd character of fixed order, it is sufficient to improve estimates for maximal sums of that same character. A converse implication to a similar effect also holds for both odd and even characters of a fixed order. 
\begin{prop} \label{Conv}
Let $q$ be large and $g\geq 2$ fixed. Let $1 > \delta > Cg\left(\frac{\log_3 q}{\log_2 q}\right)^{1/2}$ for $C>0$ sufficiently large. Let $\chi$ be a primitive character modulo $q$ of order $g$, and suppose that 
\begin{align*}
\max_{1 \leq t \leq q} \left|\sum_{n \leq t} \chi(n) \right| > \delta \sqrt{q}\log q.
\end{align*}
Then there is an absolute constant $c > 0$ such that for any $\e \geq c\delta/\log(1/\delta)$, there is a unique primitive character $\psi$ satisfying $\chi\psi(-1) = -1$ whose order $k$ satisfies $k|g$ and whose conductor $m = O(1)$ is independent of $\delta$, and a $t > q^{\e}$ such that
$$\sum_{n \leq t} \chi(n)\bar{\psi}(n) \gg_{\delta} t.$$
In particular, if for any odd primitive character $\chi'$ of conductor $q'$ and order dividing $g$ we have $\sum_{n \leq t} \chi'(n) = o(t)$ for any $t > (q')^{\e}$ and any $\e >0$, then for any odd primitive character $\chi$ of order $g$ and conductor $q$ we have $\max_{1 \leq t \leq q} \left|\sum_{n \leq t} \chi(n)\right| = o(\sqrt{q}\log q)$.
\end{prop}
\noindent See Remark \ref{RemFrGo} for an application of this proposition in regards to Burgess-type problems for even characters of a fixed (even) order $g$.\\
We deduce from Corollary \ref{Qual} the following, which for instance gives a single condition required to verify Vinogradov's conjecture for any given large prime $p \equiv 3 \pmod{4}$. \\
In the sequel, let $n_{\chi} := \min\{n \in \mb{N} : \chi(n) \neq 0,1\}$, and for $\chi = \left(\frac{\cdot}{p}\right)$, let $n_p := n_{\chi}$. 
\begin{cor} \label{TowardsVino}
Let $p$ be a large prime and $\chi$ be a primitive odd character modulo $p$ such that $$\max_{1 \leq t \le p} \left|\sum_{n \leq t} \chi(n)\right| = o_{p\ra \infty}(\sqrt{p}\log p).$$ 
Then $\log n_{\chi} = o_{p \ra \infty}(\log p)$. 
\end{cor} 
\noindent The above corollary should be compared to the main result of \cite{BoGo}, in which an \emph{a priori} stronger assumption about improvements in estimates of maximal character sums of \emph{all} even quadratic characters is used to derive the above conclusion (albeit with a nice quantitative relationship between the $n_{\chi}$ and the amount of savings in P\'{o}lya-Vinogradov inequality). 
\begin{rem} \label{RemOddOrd}
It would be nice to be able to apply a result like Theorem \ref{THMGEN} in cases where an improvement to the P\'{o}lya-Vinogradov inequality is known to be possible. Such is the case for instance when $\chi$ has odd order, as was demonstrated in a breakthrough result of Granville and Soundararajan \cite{GrSo}. In light of refinements by Goldmakher \cite{Gold} and by Lamzouri and the author \cite{LamMan}, the P\'{o}lya-Vinogradov inequality for a character $\chi$ modulo $q$ of odd order $g \geq 3$ can be improved unconditionally in the form
$$\max_{1 \leq t \leq q} \left|\sum_{n \leq t} \chi(n) \right| \ll_{\e} \sqrt{q}(\log q)^{1-\delta_g}(\log_2 q)^{-1/4+\e}$$
for any $\e > 0$, where $\delta_g := 1-\frac{g}{\pi}\sin(\pi/g).$ \\
Unfortunately, the method of proof here relies crucially on the parity of $\chi$ being odd, and odd order characters are necessarily of even parity. At best, our proof allows us to show that if $\chi$ is a character of fixed odd order and $\psi$ is any odd character of small conductor (e.g., the Legendre symbol modulo $3$) then we have cancellation in the partial sums of $\chi\psi$ on the scale $t > q^{\e}$ for any $\e > 0$ (see Remark \ref{RemGSBurg} below). However, such a result follows directly from Corollary 1.7 of \cite{GSBurg}, which admits a relatively simple proof.
\end{rem}
\subsection{Strategy of Proof}
The idea of the proof of Theorem \ref{THMGEN} is inspired by the work of Fromm and Goldmakher \cite{FrGo}, and we explain it here. Let $\xi$ be a quadratic character with prime modulus $p$ and let $\psi$ be some auxiliary quadratic character of opposite parity (and small conductor). Assuming an ``improved P\'{o}lya-Vinogradov'' hypothesis like \eqref{PVIMP} for $\xi\psi$, it can be shown that
\begin{equation}\label{LEV}
\max_{1 \leq N \leq q} \left|\sum_{1 \leq n \leq N} \frac{\xi(n)}{n}\right| \ll \frac{\log q}{f(q)}.
\end{equation}
In \cite{FrGo}, the strategy implemented is to show that if the C\'{e}saro partial sum $\sum_{n \leq t} \xi(n)$ is $\gg t$ for some $t > q^{\e}$ then this implies that the logarithmically-averaged partial sum $\left|\sum_{1 \leq n \leq t'} \frac{\xi(n)}{n}\right|$ is prohibitively large in light of \eqref{LEV}, with $t'$ of size commensurate with $t$ (in logarithmic scale). To do this they use a fact about real-valued multiplicative functions, namely that the sizes of the logarithmic and C\'{e}saro means of a real-valued function $f$ are both controlled (in terms of both upper and lower bounds) by
$$S_f(t) := \sum_{p \leq t} \frac{1-f(p)}{p}.$$
This is implied directly by a result of Hall and Tenenbaum \cite{HaT}. Crucial use in this connection is made of the fact that the function $1 \ast f$ is non-negative, and a lower bound for this non-negative function is used crucially in the proof. \\
In this paper we shall deal with characters that are not necessarily real-valued, preventing us from directly employing the above techniques.\footnote{While the main result in \cite{HaT} is applicable to complex-valued functions as well, they are for instance not well-suited to functions taking non-zero values in roots of unity of sufficiently large order. For instance, in the notation there, the best-possible constant $K$ appearing in the estimate $t^{-1}\left|\sum_{n \leq t} g(n)\right| \ll e^{-K\text{Re}(S_g(t))}$ is, for those functions $g$ considered here, necessarily $0$ (apply (5) there with $\delta = 1$ and $\phi = 0$).} However, by invoking ideas from pretentious number theory we will be follow a more elaborate proof scheme that is similar to the one described above to arrive at the desired ends. \\
To be precise, we derive a quantitative relationship between the C\'{e}saro partial sums and the logarithmic partial sums of a multiplicative function $f$ whose non-zero values are roots of unity of some fixed order $k$; see Proposition \ref{CESTOLOG} below. Upper bounds on the log-averaged sums of $\chi$ will therefore imply, according to this relationship, that the C\'{e}saro partial sums of $\chi$ cannot be too large, even when $t \asymp_{\e} q^{\e}$ (and in fact, with $t = q^{o(1)}$, as $q \ra \infty$). \\
These improved upper bounds for log-averaged sums of $\chi$ are established in Proposition \ref{KeyProp2}, which invokes the hypothesized improvement in the P\'{o}lya-Vinogradov inequality. P\'{o}lya's Fourier expansion (see Lemma 3.3 below) implies that\footnote{Here and elsewhere, given a character $\xi$ modulo $\ell$, we write the Gauss sum of $\xi$ as $\tau(\xi) := \sum_{a \pmod{\ell}} \xi(a)e^{2\pi i a/\ell}.$}
$$\max_{1 \leq t \leq q} \left|\sum_{1 \leq n \leq t} \chi(n)\right| = \frac{|\tau(\chi)|}{2\pi} \max_{\alpha \in [0,1]} \left|\sum_{1 \leq |n| \leq q} \frac{\chi(n)}{n}\left(1-e(n\alpha)\right)\right| + O(\sqrt{q}),$$
which demonstrates that the log-averaged exponential sums of $\chi$ must accordingly be small. Specializing to $\alpha = a/m$ with $m$ small in some precise sense and writing $e(na/m)$ in the basis of Dirichlet characters modulo $m$, 
it follows first that the log-averaged sum of $\chi\psi$ is small, where $\psi$ is some non-principal character with $\psi(-1) = -\chi(-1)$ (provided $m$ is chosen suitably). With some additional work (see e.g., Lemmata \ref{CASEWITHPAR} and \ref{LOGMEAN})
this can then be used to establish the required savings\footnote{This is where the assumption in Theorem \ref{THMGEN} that $\chi$ is odd is used crucially, as we must pass from a sum of $\chi(n)/n$ over $1 \leq |n| \leq q$ to a sum over $1 \leq n \leq q$. If $\chi$ were even then the two-sided sum vanishes, and we can deduce nothing about the 1-sided sum.}
$$\max_{1 \leq N \leq q} \left|\sum_{1 \leq n \leq N} \frac{\chi(n)}{n}\right| = o(\log q).$$
It should be noted that neither the bounds nor the range in $t$ in Theorem \ref{THMGEN} are expected to be optimal; it is not our main objective here to give the best such result.
\section{Relating the Logarithmic and C\'{e}saro Partial Sums of Multiplicative Functions}
In this section, we show how Lemma B of \cite{FrGo} regarding large C\'{e}saro and logarithmic sums of real-valued $1$-bounded mutliplicative functions can be adapted to deal with complex-valued Dirichlet characters of any fixed order. \\
Our first observation towards Theorem \ref{THMGEN} is the following, which will allow us to reduce matters to bounding short sums of characters whose moduli have relatively few small prime factors.
\begin{lem} \label{Reduce}
Let $\xi(t)$ be a non-decreasing function. Suppose $\chi$ is a primitive character modulo $q$, and that $\sum_{p|q} 1/p \geq \log \xi(q)$. Then for any $t > \exp(\sqrt{\log q})$ we have
\begin{align*}
\frac{1}{t}\sum_{n \leq t} \chi(n) \ll 1/\xi(t).
\end{align*}
\end{lem}
\begin{proof}
By the fundamental lemma of the sieve, there is a large constant $C > 0$ such that
\begin{align*}
\left|\frac{1}{t}\sum_{n \leq t} \chi(n)\right| &\leq \frac{1}{t}\sum_{n \leq t} |\chi(n)| = \frac{1}{t}\sum_{n \leq t \atop (n,q) = 1} 1 \ll \prod_{p \leq t^{1/C} \atop p | q} \left(1-\frac{1}{p}\right) \ll \exp\left(-\sum_{p \leq t^{1/C} \atop p|q} \frac{1}{p}\right).
\end{align*}
Note that the number of primes $p|q$ with $p > t^{1/C}$ is at most $C\frac{\log q}{\log t} \leq C\sqrt{\log q}$, so that
\begin{align*}
\sum_{p \leq t^{1/C} \atop p|q} \frac{1}{p} \geq \sum_{p|q} \frac{1}{p} - Ct^{-1/C}\sqrt{\log q} \geq \log \xi(q) - O(1),
\end{align*}
by hypothesis. Inserting this into the above yields
\begin{align*}
\left|\frac{1}{t}\sum_{n \leq t} \chi(n)\right| \ll 1/\xi(q) \leq 1/\xi(t),
\end{align*}
as claimed.
\end{proof}
In light of Lemma \ref{Reduce}, we shall mostly focus on characters whose moduli have relatively few small prime factors. To be precise, consider the collection $\mc{S}$ of non-decreasing functions $\xi(t)$ tending to infinity with $t$, such that for all large $x$, $\log \xi(x) \asymp \log \xi(t)$ for all $t \in (\exp\left(\sqrt{\log x}\right),x]$. 
Given a map $\xi \in \mc{S}$, we
define
$$\mc{Q}(\xi) := \left\{q \in \mb{N} : \sum_{p|q} \frac{1}{p} < \log \xi(q)\right\}.$$
Note that for any unbounded function $\xi$, the set of all large primes belong to $\mc{Q}(\xi)$. Moreover, by Markov's inequality, for any large $x$,
\begin{align*}
&|(\sqrt{x},x] \bk \mc{Q}(\xi)| \leq \frac{1}{\log \xi(\sqrt{x})} \sum_{\sqrt{x} < q \leq x} \left(\sum_{p|q} \frac{1}{p}\right) \ll \frac{1}{\log \xi(x)}\sum_{p \leq x} \frac{1}{p}\llf \frac{x}{p}\rrf \ll \frac{x}{\log \xi(x)},
\end{align*}
so that anyway $\mc{Q}(\xi)$ is a set of density 1 for all $\xi \in \mc{S}$. Moreover, we have a well-known uniform bound $\sum_{p|q} 1/p \leq \log\log\log q + O(1)$, which shows that if $\xi(q) \geq \log\log q$ then $\mc{Q}(\xi)$ consists of all sufficiently large moduli. \\
The main result of this section
is the following. Here and elsewhere, $\mb{U}$ denotes the closed unit disc.
\begin{prop} \label{CESTOLOG}
Let $k \geq 2$ and let $q$ be large. Let $\xi \in \mc{S}$ be such that $q \in \mc{Q}(\xi)$, and that 
\begin{equation}\label{QCONST}
\xi(t) \ll_{\e} (\log t)^{1/(27k^2)-\e}
\end{equation}
for all $t \geq \exp\left(\sqrt{\log q}\right)$. Let $f : \mb{N} \ra \mb{U}$ be a completely multiplicative function vanishing at precisely those primes dividing $q$, such that for each prime $p$ where $f(p) \neq 0$, we have $f(p)^k = 1$. \\
Then there is a constant $x_0 = x_0(\xi,k)$ such that the following is true. Suppose that 
\begin{equation}\label{XIHYP}
\left|\frac{1}{x}\sum_{n \leq x} f(n)\right| > 1/\xi(x),
\end{equation}
for some $x \geq \max\left\{x_0,e^{\sqrt{\log q}}\right\}$. \\
a) In general, we have
$$\max_{\sqrt{x} < y \leq x} \left|\frac{1}{\log y}\sum_{n \leq y} \frac{f(n)}{n}\right| + \frac{1}{\log x}\gg_k \xi(x)^{-38k^2\xi(x)^{19k^2}}.$$
b) If $k$ is odd then
$$\max_{\sqrt{x} < y \leq x} \left|\frac{1}{\log y}\sum_{n \leq y} \frac{f(n)}{n}\right| + \frac{1}{\log x} \gg_k \xi(x)^{-72k^2}.$$
\end{prop}
The proof proceeds by showing that the logarithmic partial sums of $f$ can be minorized by the C\'{e}saro average of a certain non-negative multiplicative function $g$ to be constructed below. Some work will then be done to show that when $\xi(x)$ is not too large the partial sums of $g$ are sufficiently large for this lower bound to be useful. \\
Our first lemma shows how to construct this function $g$ for the complex-valued functions that we shall consider.
\begin{lem}\label{CONVPOS}
Let $f: \mb{N} \ra \mb{U}$ be completely multiplicative, such that $f(p) \in S^1 \cup \{0\}$ for all primes $p$, and let $g := 1 \ast 1 \ast f \ast \bar{f}$. 
Then $g$ is a non-negative multiplicative function.
\end{lem}
\begin{proof}
That $g$ is multiplicative is immediate, given that it is defined from convolutions of multiplicative functions. We shall thus prove that for any prime power $p^k$, we have $g(p^k) \geq 0$. \\
Assume first that $f(p) \neq 0$, so that $|f(p)| = 1$. Write $h := 1 \ast f$. Then $g = h \ast \bar{h}$, and so we have
\begin{align*}
g(p^k) &= \sum_{0 \leq j \leq k} h(p^j) \bar{h(p^{k-j})} = \sum_{0 \leq j \leq k} \left(\sum_{0 \leq u \leq j} f(p)^u\right)\left(\sum_{0 \leq v \leq k-j} f(p)^{-v}\right) = \sum_{0 \leq u,v \leq k} f(p)^{u-v} \sum_{u \leq j \leq k-v} 1 \\
&= \sum_{0 \leq u,v \leq k} f(p)^{u-v} (k+1-(u+v)) = \sum_{|w| \leq k} f(p)^w \sum_{\max\{0,-w\} \leq v \leq \min\{k,k-w\}}(k+1 - w-2v).
\end{align*}
It is easily checked that the inner sum is $k+1-|w|$, for all $|w| \leq k$. 
Upon writing\footnote{As is standard, given $t \in \mb{R}$ we write $e(t) := e^{2\pi i t}$.} $f(p) = e(\theta_p)$, this gives
$$g(p^k) = \sum_{|w| \leq k} f(p)^w (k+1-|w|) = (k+1)K_{k+1}(\theta_p),$$
where for $N \in \mb{N}$, $$K_N(t) := \sum_{|j| \leq N-1} (1-|j|/N)e(jt) = \frac{1}{N}\left(\frac{\sin(\pi Nt)}{\sin(\pi t)}\right)^2$$ denotes the order $N$ F\'{e}jer kernel. This establishes that $g(p^k) \geq 0$ for all $k$ whenever $|f(p)| = 1$.\\
Now, suppose that $f(p) = 0$. Then $h(p^k) = 1$ for all $k\geq 0$, and thus 
$$g(p^k) = \sum_{0 \leq j \leq k} h(p^j)\bar{h(p^{k-j})} = k+1.$$ 
This completes the proof in the case $f(p) = 0$. The claim follows.
\end{proof}
\begin{lem} \label{REVTONN}
Let $f : \mb{N} \ra \mb{U}$ be completely multiplicative and such that for all primes $p$ we have $f(p) \in S^1 \cup \{0\}$. Then there is an absolute constant $t_0 \geq 2$ such that for any $t \geq t_0$,
$$\max_{\sqrt{t} \leq x \leq t} \left|\frac{1}{\log x}\sum_{n \leq x} \frac{f(n)}{n}\right| + \frac{1}{\log t}\gg \frac{1}{(\log t)^3} \frac{1}{t}\sum_{n \leq t} g(n),$$
where $g = 1 \ast 1 \ast f \ast \bar{f}$.
\end{lem}
\begin{proof}
As in the proof of the previous lemma, set $h := 1 \ast f$, so that $g = h\ast \bar{h}$. We will bound the partial sums of $h$ in two ways.
Let $t$ be sufficiently large. On one hand, for any $\sqrt{t} < x \leq t$ we have
\begin{equation}\label{TRIV}
\frac{1}{x}\sum_{n \leq x} h(n) = \frac{1}{x}\sum_{md\leq x} f(m) = \frac{1}{x}\sum_{m \leq x} f(m)\left(\frac{x}{m} + O(1)\right) = \sum_{m \leq x} \frac{f(m)}{m} + O(1).
\end{equation}
Now, the previous lemma shows that $g$ is multiplicative and non-negative. Thus, applying the hyperbola method, we get on the other hand that 
\begin{align*}
\frac{1}{t}\sum_{n \leq t} g(n) &= 2\text{Re}\left(\sum_{m \leq \sqrt{t}} \frac{\bar{h(m)}}{m} \cdot \frac{m}{t}\sum_{n \leq t/m} h(n)\right) - \left|\frac{1}{\sqrt{t}}\sum_{n \leq \sqrt{t}} h(n)\right|^2 \\
&\leq 2\left(\sum_{m \leq \sqrt{t}}\frac{|h(m)|}{m}\right) \cdot \max_{\sqrt{t} \leq x \leq t} \left|\frac{1}{x}\sum_{n \leq x} h(n)\right|.
\end{align*}
Since $|h| \leq \tau$, we easily get
$$\sum_{m \leq \sqrt{t}} \frac{|h(m)|}{m} \leq \sum_{P^+(m) \leq \sqrt{t}} \frac{|h(m)|}{m} \ll \exp\left(\sum_{p \leq t} \frac{|h(p)|}{p}\right) \ll (\log t)^2.$$
As such, we get that
$$\max_{\sqrt{t} < x \leq t}\left|\frac{1}{x}\sum_{n \leq x} h(n)\right| \gg \frac{1}{(\log t)^2} \frac{1}{t}\sum_{n \leq t} g(n).$$
Combining this with \eqref{TRIV}, and dividing both sides by $\log t$, proves the claim.
\end{proof}
In order to establish Proposition \ref{CESTOLOG}, we will need to be able to show that the partial sums of $g$ are not too small. To this end, we will make use of the notion of pretentiousness of multiplicative functions, due originally to Granville and Soundararajan (for a detailed account of this theory, see \cite{GrSoBOOK}). A key object in this connection is the so-called \emph{pretentious distance}, which we define as follows. For $f_1,f_2 : \mb{N} \ra \mb{U}$ multiplicative and $x \geq 2$, set
$$\mb{D}(f_1,f_2;x) := \left(\sum_{p \leq x} \frac{1-\text{Re}(f_1(p)\bar{f_2(p)})}{p}\right)^{\frac{1}{2}}.$$
Moreover, for $T \geq 0$, we define\footnote{Henceforth, for convenience, given $t \in \mb{R}$ we will write $n^{it}$ to denote the multiplicative function mapping $n \mapsto n^{it}$.} $$\mc{D}_f(x;T) := \min_{|t| \leq T}\mb{D}(f,n\mapsto n^{it};x)^2.$$
Roughly speaking, a hypothesis like \eqref{XIHYP} implies (via Hal\'{a}sz' theorem, see e.g., \eqref{HMT} below) that $\mc{D}_f(x;(\log x)^2)$ is small (in a manner depending on $\xi$), and thus $f$ is $n^{it_0}$-pretentious for some $t_0 \in [-(\log x)^2,(\log x)^2]$. An important feature of multiplicative functions whose values are bounded order roots of unity is that we can guarantee that $t_0 = 0$. The following lemma gives a quantitative relation in this direction.
\begin{lem} \label{EQUIV}
Let $k \in \mb{N}$ and let $1 \leq k,T \leq x$. For any multiplicative function\footnote{Given a positive integer $m$, we write $\mu_m$ to denote the set of $m$th order roots of unity.} $f : \mb{N} \ra \mu_k$,
\begin{equation*}
\min_{|t| \leq T} \mb{D}(f,n^{it};x) \geq \frac{1}{2k}\min\left\{\sqrt{\log_2 x}, \mb{D}(f,1;x)\right\} - O(1).
\end{equation*}
\end{lem}
\begin{proof}
This is Lemma 3.1 in \cite{Rig2}. In the statement there, the $O(1)$ term depends on $k$, but following the proof it is easy to verify that it can in fact be taken absolute.
\end{proof}
\begin{lem}\label{PRETDISTEST}
Let $k \geq 2$, let $\xi$ be as in the statement of Proposition \ref{CESTOLOG}, and suppose that $q \in \mc{Q}(\xi)$. Let $f: \mb{N} \ra \mu_k \cup \{0\}$ be a completely multiplicative function such that $f(p) = 0$ if, and only if, $p|q$. 
Then there is an $x_0 = x_0(\xi,k)$ such that the following holds. If $x \geq \max\{x_0,e^{\sqrt{\log q}}\}$ and
$$\left|\frac{1}{x}\sum_{n \leq x} f(n)\right| > \frac{1}{\xi(x)},$$
then we have
$$\mb{D}(f,1;x)^2 \leq 27k^2\left(\log (\xi(x)\log \xi(x)) +O(1)\right).$$
\end{lem}
\begin{proof}
The result is immediate if $\mb{D}(f,1;x)^2 \leq 200k^2$, so in the sequel we assume that $\mb{D}(f,1;x)^2 > 200k^2$. \\
By the Hal\'{a}sz-Montgomery-Tenenbaum inequality (see Theorem III.4.6 in \cite{Ten}), for any $T \geq 1$ we have
\begin{equation}\label{HMT}
\frac{1}{x}\sum_{n \leq x} f(n) \ll \mc{D}_f(x;T)e^{-\mc{D}_f(x;T)} + \frac{1}{\sqrt{T}}.
\end{equation}
Let $T=\log x$. Let $\tilde{f}: \mb{N} \ra S^1$ be the completely multiplicative function defined on primes by 
$$\tilde{f}(p) := \begin{cases} f(p) &\text{ if $p\nmid q$}\\ 1 &\text{ if $p|q$;}\end{cases}$$ 
note that $\tilde{f}^k = 1$. For any fixed $t \in [-T,T]$, the triangle inequality for $\mb{D}$ gives
\begin{equation}\label{TRINQ}
\mb{D}(f,n^{it};x) \geq \mb{D}(\tilde{f},n^{it};x) - \mb{D}(f,\tilde{f};x) = \mb{D}(\tilde{f},n^{it};x) - \left(\sum_{p|q} \frac{1}{p}\right)^{\frac{1}{2}},
\end{equation}
while we trivially have
\begin{align} \label{TRIVQ}
\mb{D}(\tilde{f},n^{it};x) = \left(\sum_{p \leq x} \frac{1-\text{Re}(f(p)p^{-it})}{p} - \sum_{p\leq x \atop p|q} \frac{\cos(t\log p)}{p}\right)^{\frac{1}{2}} \geq \mb{D}(f,n^{it};x) - \left(\sum_{p|q}\frac{1}{p}\right)^{\frac{1}{2}}.
\end{align}
By Lemma \ref{EQUIV}, we have
$$\min_{|t| \leq T} \mb{D}(\tilde{f},n^{it};x) \geq \frac{1}{2k} \min\{\sqrt{\log_2 x},\mb{D}(\tilde{f},1;x)\} - O(1).$$
Using \eqref{TRINQ} with $t = t_0$ the minimizer of $t\mapsto \mb{D}(f,n^{it};x)$ in $[-T,T]$ and \eqref{TRIVQ} with $t = 0$, we deduce that
$$\min_{|t| \leq T} \mb{D}(f,n^{it};x) \geq \frac{1}{2k}\min\left\{\sqrt{\log_2 x},\mb{D}(f,1;x)\right\} - \frac{2k+1}{2k}\left(\sum_{p|q} \frac{1}{p}\right)^{\frac{1}{2}} - O(1),$$
and hence by the Cauchy-Schwarz inequality (and using $k \geq 2$),
$$\mc{D}_f(x;T) = \min_{|t| \leq T} \mb{D}(f,n^{it};x)^2 \geq \frac{1}{8k^2}\min\{\log_2 x, \mb{D}(f,1;x)^2\} - 2\sum_{p|q} \frac{1}{p} - O(1).$$
Plugging this into \eqref{HMT} and using $1/T < (2\xi(x))^{-2}$, we get upon rearranging that 
\begin{equation} \label{ONEOVERXI}
\frac{1}{\xi(x)} \ll \frac{1}{8k^2}\min\{\log_2 x,\mb{D}(f,1;x)^2\}\exp\left(-\frac{1}{8k^2}\min\{\log_2 x,\mb{D}(f,1;x)^2\} + 2\sum_{p|q} \frac{1}{p}\right).
\end{equation}
Note that the arithmetic function $d \mapsto \sum_{p|d} 1/p$ has maximal order\footnote{This is witnessed by a primorial $d = \prod_{p \leq \log d'} p$, with $d' = (1+o(1))\log d$ (by the prime number theorem).} $\log_3 x +O(1)$, for $d \leq x$.  
Thus, since $\xi(x) \ll (\log x)^{\frac{1}{10k^2}}$, it follows from \eqref{ONEOVERXI} that $\min\{\log_2 x,\mb{D}(f,1;x)^2\} = \mb{D}(f,1;x)^2$. Hence,
$$\frac{1}{\xi(x)} \ll \frac{1}{8k^2} \mb{D}(f,1;x)^2e^{-\frac{1}{8k^2}\mb{D}(f,1;x)^2} \cdot \exp\left(2\sum_{p|q}\frac{1}{p}\right).$$
Taking logarithms, and noting that  $\log X < X/17$ for $X = \mb{D}(f,1;x)^2/(8k^2) > 25$, this implies that 
$$\mb{D}(f,1;x)^2 \leq 8.5k^2\left(\log(\xi(x)\log \xi(x)) + O(1)+2\sum_{p|q}\frac{1}{p}\right) \leq 27k^2\left(\log(\xi(x)\log \xi(x)) + O(1)\right),$$
using $\sum_{p|q} 1/p < \log \xi(x)$, on account of our assumption $q \in \mc{Q}(\xi)$.
\end{proof}

%
\begin{proof}[Proof of Proposition \ref{CESTOLOG}]
In light of Lemma \ref{REVTONN}, it suffices to derive a lower bound for $\sum_{n \leq x} g(n)$ in both parts a) and b) (recalling that $g = 1\ast 1\ast f \ast \bar{f}$).\\
a) 
By Theorem 1.2 of \cite{Hil}, we get
\begin{equation}\label{HilApp}
\frac{1}{x}\sum_{n \leq x} g(n) \gg \exp\left(\sum_{p \leq x} \frac{g(p)-1}{p}\right) \sg_-\left(\exp\left(\sum_{p \leq x} \frac{\max\{0,1-g(p)\}}{p}\right)\right) + O\left(\exp\left(-(\log x)^{\beta}\right)\right),
\end{equation}
for some absolute constant $\beta > 0$. Here, we have $\sg_-(u) := u\rho(u)$, where $\rho$ is the Dickman-de Bruijn function, which satisfies $\rho(u) \gg u^{-3u/2}$ for large enough $u$ (see, for instance, equation (1.12) in \cite{Gra}). Now, observe that $g(p) = 2(1+\text{Re}(f(p)))$ for all primes $p$. Thus, by Lemma \ref{PRETDISTEST}, we get
\begin{align*}
\sum_{p \leq x} \frac{g(p)-1}{p} &= \sum_{p \leq x} \frac{1+2\text{Re}(f(p))}{p} = 3\log_2 x - 2\mb{D}(f,1;x)^2 + O(1) \\
&\geq 3\log_2 x - 54k^2\log (\xi(x)\log \xi(x)) -O_k(1).
\end{align*}
Note that $\max\{0,1-g(p)\} \neq 0$ if, and only if, $\text{Re}(f(p)) < -1/2$, in which case the numerator is bounded by 1. Hence, for $q$ (and thus $x$) sufficiently large, we get
\begin{align*}
\log u &:= \sum_{p \leq x} \frac{\max\{0,1-g(p)\}}{p} < \frac{2}{3}\sum_{p \leq x} \frac{1-\text{Re}(f(p))}{p} \\
&\leq \frac{54k^2}{3}\log(\xi(x)\log \xi(x)) + O_k(1) \leq 19k^2\log \xi(x),
\end{align*}
From \eqref{HilApp}, it follows that for $x$ sufficiently large (in terms of $\xi$),
$$\frac{1}{x}\sum_{n \leq x} g(n)\gg_k (\log x)^3 \xi(x)^{-54k^2} u^{-3u/2} \gg (\log x)^3 \xi(x)^{-38k^2\xi(x)^{19k^2}}.$$
By Lemma \ref{REVTONN}, this results in the estimate
$$\max_{\sqrt{x} < y \leq x} \left|\frac{1}{\log y} \sum_{n \leq y} \frac{f(n)}{n}\right| + \frac{1}{\log x}\gg \xi(x)^{-38k^2\xi(x)^{19k^2}},$$
as claimed. \\
b) We assume throughout the proof that $x > \exp(\sqrt{\log q})$. We would like to apply part ii) of Th\'{e}or\`{e}me 1.1 in \cite{Ten2}, which gives a better lower bound for partial sums of $g$ compared to Hildebrand's estimate provided $g$ satisfies various conditions. To this end we verify the conditions i)-v) stipulated there. Since $|f| \leq 1$, it follows that $g(p) \leq 4$ for all $p$, and moreover $g(n) \ll_{\e} n^{\e}$ for all $n$. Thus, conditions i) and ii) for that theorem, namely
\begin{align*}
&\sum_{p \leq y} g(p)\log p \leq Ay \text{ for all $y \geq 2$} \\ 
&\sum_p \sum_{\nu \geq 2} \frac{g(p^{\nu})\log(p^{\nu})}{p^{\nu}} \leq B,
\end{align*}
both hold with $A = 5$ and $B = O(1)$ by the prime number theorem. Moreover, condition iii) holds since for each $\delta > 0$, setting $\e = \delta/5$ and taking $y$ sufficiently large in terms of $\delta$ (in line with the prime number theorem), we get
$$\sum_{y < p \leq y(1+\e)} g(p) \log p \leq \delta y$$
for all $y \geq y_0$.
Turning to the first condition in v) (in light of its application in the proof of the theorem), it is enough to check that for any $\lambda > 0$ there is $\eta > 0$ such that for any choice of $0 < \sg < \tau < 1-\sg$, we have
$$\sum_{x^{\tau}<p \leq x^{1-\sg}}\frac{1_{g(p)\leq \eta}}{p} \leq \lambda.$$
To see this, we note that since $k$ is odd, writing $k = 2m+1$, it follows that 
$$|f(p)-1|^2 \leq 2(1-\cos(2\pi m/(2m+1))) \leq 4-2\frac{\pi^2}{3(2m+1)^2} = 4-\frac{2\pi^2}{3k^2},$$ 
for any prime $p$.
Hence, selecting $\eta = \frac{\pi^2}{3k^2}$, we see that since 
$$g(p) = \begin{cases} 2 &\text{ if $f(p) = 0$}\\ 4-|f(p)-1|^2 &\text{ otherwise,}\end{cases}$$ 
we find $g(p) > \eta$ for all $p$. Hence, condition v) there holds with this $\eta$ for any $\lambda > 0$ and any choice of $\sg,\tau$ as above. \\
It remains to demonstrate iv) in the notation there. We pick $\sg = \xi(x)^{-c_0}$, for $c_0 > 0$ a constant to be chosen, and $\tau = 1/46$. Putting $h := (1-\tau)/A\tau = 9$, it follows that $Q(h) := h\log h - h+1 \geq 1/2$. It now suffices to check that
$$\sum_{x^{\sg} < p \leq x^{\tau}}\frac{g(p)}{p} \geq \log\left(\frac{1+\eta}{1-e^{-5Q(h)}}\right)$$
with the above choices made. To see this, note that 
$$\sum_{x^{\sg} < p \leq x^{\tau}} \frac{g(p)}{p} = \sum_{x^{\sg} < p \leq x^{\tau} \atop p\nmid q} \frac{g(p)}{p} + O\left(x^{-\sg}\frac{\log q}{\sg \log x}\right),$$
the error term arising from $g(p)\leq 4$ for all primes $p$, and thus
$$\sum_{p|q \atop p > x^{\sg}} \frac{g(p)}{p} \leq 4x^{-\sg}\sum_{p|q \atop p > x^{\sg}} 1 \leq 4x^{-\sg} \frac{\log q}{\log(x^{\sg})}.$$
For the above sum, noting that $|f(p)-1|^2 = 2(1-\text{Re}(f(p)))$ whenever $p\nmid q$, we have
\begin{align*}
\sum_{x^{\sg} < p \leq x^{\tau} \atop p\nmid q} \frac{g(p)}{p} &\geq 4\sum_{x^{\sg} < p \leq x^{\tau}} \frac{1}{p} - 2\left(\mb{D}(f,1;x^{\tau})^2 - \mb{D}(f,1;x^{\sg})^2\right) - O\left(x^{-\sg}\frac{\log q}{\sg \log x}\right) \\
&\geq 4\log(\tau\sg^{-1}) -2\mb{D}(f,1;x)^2 - O\left(\frac{x^{-\sg}\log q}{\sg\log x}\right).
\end{align*}
In light of Lemma \ref{PRETDISTEST},
choosing $c_0 = 27k^2$ gives
$$2\mb{D}(f,1;x)^2 \leq 54k^2\log(\xi(x)\log \xi(x)) + O_k\left(1\right) \leq 2\log(\tau \sg^{-1}) + O_k\left(1\right).$$
As \eqref{QCONST} gives $\sg \log x \gg_{\e} (\log x)^{\e}$ and thus $x^{-\sg} \log q \leq x^{-\sg}(\log x)^2 \ll 1$, we get that if $x$ is sufficiently large in terms of $k$ then
\begin{align*}
\sum_{x^{\sg} < p \leq x^{\tau} \atop p\nmid q} \frac{g(p)}{p} &\geq 2\log(\tau\sg^{-1}) - O_k\left(1\right) \geq 54 k^2\log \xi(x) - O_k(1) \\
&\geq 2\left(\eta + e^{-5Q(h)}\right) \geq\log\left(\frac{1+ \eta}{1-e^{-5Q(h)}}\right).
\end{align*}
Hence, iv) holds in the theorem in \cite{Ten2}, provided that $x$ is sufficiently large (in terms of $\xi$ and $k$ alone). Th\'{e}or\`{e}me 1.1 there then implies\footnote{The result appearing there is given with a constant that is not made implicit in terms of the parameters $\sg$ and $\eta$. However, the whole argument in Section 2 of that paper can be made to depend explicitly on those two parameters, which may then be allowed to depend on $x$. Following the calculation there, one readily finds that, provided that $\sg \log x \geq C\log\log x$ for some sufficiently large $C > 0$,
$$\sum_{n \leq x} g(n) \gg \frac{x}{\log x} \exp\left(\sum_{p\leq x} g(p)/p\right)\left(\frac{\eta^2\sg}{1+\tau}e^{-2A\tau}\left(1-e^{-AQ(h)}\right) - O_{h,A}\left(\frac{\eta}{\log x} + \tau(\log x)e^{-\sqrt{\tau \log x}}\right)\right),$$
whenever $g: \mb{N} \ra \mb{C}$ is non-negative, multiplicative and satisfies conditions i)-v) above with $A > 0$, $\eta > 0$, $0 < \sg < \tau < \min\{1-\sg,1/(1+A)\}$ and $h = (1-\tau)/\tau A$.}
that for any $x \geq \max\{x_0(\xi,k),e^{\sqrt{\log q}}\}$ (recalling that $\eta \asymp 1/k^2$)
\begin{align*}
\frac{1}{x}\sum_{n \leq x} g(n) &\gg\frac{1}{k^4} \xi(x)^{-17k^2}\frac{x}{\log x}\exp\left(\sum_{p \leq x} \frac{g(p)}{p}\right) \\
&= \frac{1}{k^4}\xi(x)^{-17k^2}\frac{x}{\log x}\exp\left(\sum_{p \leq x} \frac{4-2(1-\text{Re}(f(p)))}{p}\right) \\
&\gg \frac{1}{k^4}\xi(x)^{-17k^2}x(\log x)^3 e^{-2\mb{D}(f,1;x)^2} \gg \frac{1}{k^4}x(\log x)^3\xi(x)^{-72k^2} \\
&\gg \frac{1}{k^4}x(\log x)^3\xi(x)^{-72k^2},
\end{align*}
in light of Lemma \ref{PRETDISTEST} and our assumption on $\xi(x)$.
Combining this with Lemma \ref{REVTONN}, we get that
$$\max_{\sqrt{x} < y \leq x} \left|\frac{1}{\log y}\sum_{n \leq y} \frac{f(n)}{n}\right| +\frac{1}{\log x} \gg \frac{1}{k^4}\xi(x)^{-72k^2},$$
and the claim follows.
\end{proof}
\section{Proof of Theorem \ref{THMGEN}}
Let $q$ be large and let $\chi$ be a primitive odd character modulo $q$. Throughout this section, we shall assume that there is a non-decreasing function $a(t) \ra \infty$ as $t \ra \infty$, satisfying $\log a(t) \asymp \log a(q)$ and $a(t) \leq \log t$ for all $t > \exp\left(\sqrt{\log q}\right)$, and such that the P\'{o}lya-Vinogradov inequality for $\chi$ can be improved in the form
\begin{equation}\label{BETTERPV}
\max_{1 \leq t \leq q} \left|\sum_{n \leq t} \chi(n)\right| \leq \sqrt{q}\frac{\log q}{a(q)}.
\end{equation}
Our goal is to show that such a hypothesis implies Theorem \ref{THMGEN}, regarding bounds for $\sum_{n \leq t} \chi(n)$ whenever $t > q^{\e}$. \\
\begin{prop}\label{KeyProp2}
Let $q \geq 3$ and let $\chi$ be an odd primitive character modulo $q$. Assume that $\chi$ satisfies \eqref{BETTERPV}. Then for any $\e > 0$, 
\begin{align*}
\max_{1 \leq N \leq q} \left|\sum_{1 \leq n \leq N} \frac{\chi(n)}{n} \right| \ll_{\e} (\log q)a(q)^{-1/3+\e}.
\end{align*}
\end{prop}
To prove Proposition \ref{KeyProp2}, we first recall the following simple harmonic analysis device which is essentially due to Goldmakher and Lamzouri (based on work of Paley).
\begin{lem}[cf. Lemma 2.2 of \cite{GolLam}] \label{Paley}
Let $A > 0$. Let $a: \mb{Z} \ra \mb{C}$ be a sequence with $\sup_{n \geq 1}|a(n)| \leq A$. Then
\begin{align*}
\max_{\alpha \in [0,1]} \max_{1 \leq N \leq q} \left|\sum_{1 \leq |n| \leq N} \frac{a(n)}{n}e(n\alpha)\right| = \max_{\alpha \in [0,1]} \left|\sum_{1 \leq |n|\leq q} \frac{a(n)}{n}e(n\alpha)\right| + O(A).
\end{align*}
\end{lem}
\begin{proof}
Writing $a_n = Aa'_n$ for all $n \in \mb{Z}$, we find that $\{a_n'\}_n \subset \mb{U}$, and the claim then follows immediately from Lemma 2.2 of \cite{GolLam}.
\end{proof}
The next lemma allows us to bound maximal C\'{e}saro partial sums of $\chi$ from below in terms of maximal logarithmically averaged partial sums of the twist $\chi\psi$, provided that $\chi$ and $\psi$ have opposite parity. This is a variant of a result that has appeared before in several places (including, among others \cite{BoGo}, \cite{LamMan}) but with no fixed restrictions on the parity of $\psi$ in particular (previous variants used $\psi$ odd, though we shall require $\psi$ even).
\begin{lem} \label{CASEWITHPAR}
Let $\chi$ and $\psi$ be primitive characters to respective moduli $q,\ell \geq 2$, such that $\chi\psi$ is odd. Then
$$\max_{1 \leq t \leq q} \left|\sum_{n \leq t} \chi(n)\right| +O(\sqrt{q}) \geq \frac{\sqrt{q\ell}}{2\pi\phi(\ell)} \max_{1 \leq N \leq q} \left|\sum_{1 \leq n \leq N} \frac{\chi(n)\psi(n)}{n}\right|.$$
\end{lem}
\begin{proof}
By the P\'{o}lya Fourier expansion (see, e.g., (2.1) in \cite{LamMan}), we have
\begin{align*}
\max_{1 \leq t \leq q} \left|\sum_{n \leq t} \chi(n)\right| &= \frac{\sqrt{q}}{2\pi}\max_{1 \leq t \leq q} \left|\sum_{1 \leq |n| \leq q} \frac{\bar{\chi}(n)}{n} \left(1-e\left(-\frac{nt}{q}\right)\right)\right| + O\left(\log q \right) \\
&= \frac{\sqrt{q}}{2\pi}\max_{\theta \in [0,1]} \left|\sum_{1 \leq |n| \leq q} \frac{\chi(n)}{n}\left(1-e(n\theta)\right)\right| + O(\sqrt{q}),
\end{align*}
where the second equality follows since if the sum is maximized at $\theta_0 \in [0,1]$, we may take $t_0 := \llf q\theta_0\rrf$ and thus
\begin{align*}
&\left|\left|\sum_{1 \leq |n| \leq q} \frac{\chi(n)}{n}\left(1-e\left(\frac{nt_0}{q}\right)\right)\right| - \left|\sum_{1 \leq |n| \leq q} \frac{\chi(n)}{n}\left(1-e(n\theta_0)\right)\right|\right| \\
&\leq \sum_{n \leq q} \frac{1}{n}|e(n\theta_0) - e(nt_0/q)| \leq 2\pi q|\theta_0-t_0/q| \ll 1.
\end{align*}
By the triangle inequality, we observe that
\begin{align*}
\max_{\theta \in [0,1]} \left|\sum_{1 \leq |n| \leq q} \frac{\chi(n)}{n}(1-e(n\theta))\right| \geq \frac{1}{2}\max_{\theta\in [0,1]} \max_{\alpha\in [0,1]} \left|\sum_{1 \leq |n| \leq q} \frac{\chi(n)}{n}(e(n\alpha)-e(n(\alpha+\theta)))\right|,
\end{align*}
and by Lemma \ref{Paley} (with $a(n) := \chi(n)(1-e(n\theta)$), the RHS of this last expression is
\begin{align*}
&\frac{1}{2}\max_{\theta\in [0,1]} \max_{\alpha \in [0,1]} \max_{1 \leq N \leq q} \left|\sum_{1 \leq |n| \leq N} \frac{\chi(n)}{n}(1-e(n\theta))e(n\alpha)\right| +O(1) \\
&\geq \frac{1}{2}\max_{\theta \in [0,1]} \max_{1 \leq N \leq q} \left|\sum_{1 \leq |n| \leq N} \frac{\chi(n)}{n}(1-e(n\theta))\right| + O(1).
\end{align*}
Now, on one hand, we have
\begin{align}
\max_{1 \leq N \leq q} \left|\frac{1}{\phi(\ell)}\sum_{b(\ell)} \bar{\psi}(b)\sum_{1 \leq |n| \leq N} \frac{\chi(n)}{n}\left(1-e\left(\frac{nb}{\ell}\right)\right)\right| &\leq \frac{1}{\phi(\ell)} \sum_{b (\ell)} \max_{1 \leq N \leq q} \left|\sum_{1 \leq |n| \leq N} \frac{\chi(n)}{n}\left(1-e(nb/\ell)\right)\right| \nonumber \\
&\leq \max_{\theta \in [0,1]} \max_{1 \leq N \leq q} \left|\sum_{1 \leq |n| \leq N} \frac{\chi(n)}{n}\left(1-e(n\theta)\right)\right|.\label{TRICK1}
\end{align}
On the other, as $\chi\xi(-1) = -1$ and $\psi$ is non-principal, orthogonality modulo $\ell$ gives
\begin{align}
\frac{1}{\phi(\ell)} \sum_{b(\ell)}\bar{\psi}(b)\sum_{1 \leq |n| \leq N} \frac{\chi(n)}{n}\left(1-e\left(\frac{nb}{\ell}\right)\right) &= -\frac{1}{\phi(\ell)}\sum_{1 \leq |n| \leq N} \frac{\chi(n)}{n}\cdot \frac{1}{\phi(\ell)}\sum_{b(\ell)} \bar{\psi}(b)e(nb/\ell) \nonumber\\
&= -\frac{2\tau(\bar{\psi})}{\phi(\ell)} \sum_{1 \leq n \leq N} \frac{\chi(n)\psi(n)}{n}, \label{TRICK2}
\end{align}
for each $1 \leq N \leq q$. It follows that
$$\max_{1 \leq t \leq q} \left|\sum_{n \leq t} \chi(n)\right| \geq \frac{\sqrt{q\ell}}{2\pi\phi(\ell)} \max_{1 \leq N \leq q} \left|\sum_{1 \leq n \leq N} \frac{\chi(n)\psi(n)}{n}\right| - O\left(\sqrt{q}\right),$$
and the claim follows.
\end{proof}
\begin{cor}\label{BETTERLOGMEANFORCHI}
Let $q \geq 3$ and let $\ell \geq 5$ be prime. Let $\chi$ be an odd primitive character modulo $q$, and let $\psi$ be an even primitive character modulo $\ell$. Then
\begin{align*}
\max_{1 \leq N \leq q} \left|\sum_{1 \leq n \leq N} \frac{\chi(n)}{n}\right| \ll \sqrt{\ell}\max_{\alpha \in [0,1]} \max_{1 \leq N \leq q} \left|\sum_{1 \leq |n| \leq N} \frac{\chi(n)\psi(n)}{n}e(n\alpha)\right|.
\end{align*}
\end{cor}
\begin{proof}
Following the proof of Lemma \ref{CASEWITHPAR}, particularly \eqref{TRICK1} and \eqref{TRICK2}, we deduce that
\begin{align}
&\max_{\alpha \in [0,1]} \max_{1 \leq N \leq q}\left|\sum_{1 \leq |n| \leq N} \frac{\chi(n)\psi(n)}{n}e(n\alpha)\right| 
\geq \frac{1}{\phi(\ell)} \max_{1 \leq N \leq q} \left|\sum_{1 \leq |n| \leq N} \frac{\chi(n)\psi(n)}{n} \sum_{b(\ell)}\psi(b)e(nb/\ell)\right| \nonumber\\
&= \frac{\sqrt{\ell}}{\phi(\ell)} \max_{1 \leq N \leq q} \left|\sum_{1 \leq |n| \leq N} \frac{\chi(n)\psi(n)\bar{\psi(n)}}{n}\right| = \frac{\sqrt{\ell}}{\phi(\ell)} \max_{1 \leq N \leq q} \left|\sum_{1 \leq |n| \leq N} \frac{\chi(n)}{n}1_{(n,\ell) = 1}\right| \nonumber\\
&= 2\frac{\sqrt{\ell}}{\phi(\ell)}\max_{1 \leq N \leq q} \left|\sum_{1 \leq n \leq N} \frac{\chi(n)1_{(n,\ell) = 1}}{n}\right|, \label{TRICKS}
\end{align}
using the fact that $\chi(-1) = -1$ in the last line. 
Now, since $\ell$ is prime, we observe that
\begin{align*}
\max_{1 \leq N \leq q} \left|\sum_{1 \leq n \leq N} \frac{\chi(n)}{n}\right| &\leq \max_{1 \leq N \leq q}\sum_{j \geq 0} \frac{1}{\ell^j}\left|\sum_{n \leq N/\ell^j} \frac{\chi(n)}{n}1_{(n,\ell) = 1}\right| \leq \left(\sum_{j \geq 0} \ell^{-j}\right) \max_{1 \leq N' \leq q} \left|\sum_{1 \leq n \leq N'} \frac{\chi(n)}{n}1_{(n,\ell) = 1}\right| \\
&\ll \max_{1 \leq N' \leq q} \left|\sum_{1 \leq n \leq N'} \frac{\chi(n)}{n}1_{(n,\ell) = 1}\right|.
\end{align*}
Combined with the previous estimate, this completes the proof.
\end{proof}
\begin{lem} \label{LOGMEAN}
Let $m \geq 3$. Let $\psi$ be an even non-principal character modulo $m$. As above, assume that $\chi$ is an odd primitive character modulo $q$ satisfying \eqref{BETTERPV}. Then
\begin{align*}
\max_{1 \leq N \leq q} \left|\sum_{1 \leq n \leq N} \frac{\chi(n)\psi(n)}{n}\right| \ll \tau(m)\sqrt{m}\frac{\log q}{a(q)}.
\end{align*}
\end{lem}
\begin{proof}
Factor $\psi = \psi^{\ast} \psi_0$, where $\psi^{\ast}$ is a primitive character modulo $m^{\ast}$ and $\psi_0$ is principal modulo $m/m^{\ast}$. We then have that
\begin{align*}
\max_{1 \leq N \leq q} \left|\sum_{1 \leq n \leq N} \frac{\chi(n)\psi(n)}{n}\right| \leq \sum_{d|m/m^{\ast}} \frac{1}{d}\max_{1 \leq N \leq q} \left|\sum_{1 \leq n \leq N/d} \frac{\chi(n)\psi^{\ast}(n)}{n}\right| \ll \tau(m)\max_{1 \leq N' \leq q} \left|\sum_{1 \leq n \leq N'} \frac{\chi(n)\psi^{\ast}(n)}{n}\right|.
\end{align*}
Since $\psi^{\ast}$ has the same parity as $\psi$, Lemma \ref{CASEWITHPAR} and \eqref{BETTERPV} give
\begin{align*}
\max_{1 \leq N' \leq q} \left|\sum_{1 \leq n \leq N'} \frac{\chi(n)\psi^{\ast}(n)}{n}\right| &\ll \frac{\phi(m^{\ast})}{\sqrt{m^{\ast}}} \max_{1 \leq t \leq q} \frac{1}{\sqrt{q}}\left|\sum_{1 \leq n \leq t} \chi(n)\right| \ll \sqrt{m}\frac{\log q}{a(q)},
\end{align*}
and the claim follows upon combining these two statements.
\end{proof}
\begin{lem} \label{EXPSUMBD}
Let $\chi$ be as above, and let $\psi$ be a primitive even character modulo $\ell$. Then
$$\max_{\alpha \in [0,1]} \max_{1 \leq N \leq q} \left|\sum_{1 \leq |n| \leq N} \frac{\chi(n)\psi(n)}{n}e(n\alpha)\right| \ll_{\e} \ell^{\e} (\log q)a(q)^{-1/3+\e}.$$
\end{lem}
\begin{proof}
By Lemma \ref{Paley}, we see that the LHS here is
$$\max_{\alpha\in [0,1]} \left|\sum_{1 \leq |n| \leq q} \frac{\chi(n)\psi(n)e(n\alpha)}{n}\right| + O(1).$$
Let $R = (\log q)^5$, and let $1 \leq M \leq R$ be a parameter to be chosen. By the Dirichlet approximation theorem, for any $\alpha \in [0,1]$ we can find a rational $b/r$ such that $|\alpha-b/r| \leq 1/(rR)$. We will say that $\alpha$ is \emph{minor arc} if $M < r \leq R$, and \emph{major arc} if $1 \leq r \leq M$. \\
Suppose first that $\alpha$ is minor arc. By Corollary 2.2 of \cite{Gold}, we get
$$\sum_{1 \leq |n| \leq q} \frac{\chi(n)\psi(n)}{n} e(n\alpha) \ll \frac{(\log M)^{5/2}}{\sqrt{M}}(\log q) + \log R + \log_2 q \ll \frac{(\log M)^{5/2}}{\sqrt{M}}\log q + \log_2 q.$$
Now, suppose $|\alpha-b/r| \leq 1/(rR)$, with $r \leq M$. Then applying equation (7.2) in \cite{LamMan} with  $N':= \min\{q,|r\alpha-b|^{-1}\}$, we have
\begin{align*}
\sum_{1 \leq |n| \leq q}\frac{\chi(n)\psi(n)}{n}e(n\alpha) &= \sum_{1 \leq n \leq N'} \frac{\chi(n)\psi(n)}{n}e(nb/r) + O\left(\frac{(\log R)^{3/2}}{\sqrt{R}}(\log q)^2 + \log_2 q\right) \\
&= \sum_{1 \leq |n| \leq N'} \frac{\chi(n)\psi(n)}{n}e(bn/r) + O\left(\log\log q\right).
\end{align*}
We now invoke the Granville-Soundararajan identity (see Proposition 2.3 of \cite{GrSo}), which reads
\begin{align}
&\sum_{1 \leq |n| \leq N'} \frac{\chi(n)\psi(n)}{n}e(bn/r) \nonumber\\
&= \sum_{d|r} \frac{\chi(d)\psi(d)}{d} \cdot \frac{1}{\phi(r/d)} \sum_{\psi' \text{ mod } r/d}(1-\chi\psi(-1)\psi'(-1)) \tau(\psi') \bar{\psi'}(b) \cdot \sum_{1 \leq |n| \leq N'/d} \frac{\chi(n)\psi(n)\bar{\psi'}(n)}{n} \nonumber\\
&= 2\sum_{d|r} \frac{\chi(d)\psi(d)}{d} \cdot \frac{1}{\phi(r/d)} \sum_{\psi' \text{ mod } r/d \atop \psi'(-1) = 1}\tau(\psi') \bar{\psi'}(b) \cdot \sum_{1 \leq |n| \leq N'/d} \frac{\chi(n)\psi(n)\bar{\psi'}(n)}{n},\label{GRSOID}
\end{align}
as $\chi\psi(-1) = -1$. \\
Now, $\psi\bar{\psi'}$ has modulus $\leq \ell r$,
and $\psi\bar{\psi'}$ is even, thus of opposite parity to $\chi$.
%
%
Hence, applying Lemma \ref{LOGMEAN} to the inner sums in \eqref{GRSOID} gives
\begin{align*}
\left|\sum_{1 \leq |n| \leq N'} \frac{\chi(n)\psi(n)}{n} e(bn/r)\right| &\ll \sum_{d|r} \frac{\sqrt{r/d}}{d\phi(r/d)}\sum_{\psi'\pmod{r/d} \atop \psi'(-1) =1} \max_{1 \leq N'' \leq q} \left|\sum_{1 \leq n \leq N''} \frac{\chi(n)\psi(n)\bar{\psi'}(n)}{n}\right| \\
&\ll \tau(\ell r)\sqrt{\ell r}\frac{\log q}{a(q)}\cdot \sum_{d|r} \frac{\sqrt{r/d}}{d} \ll_{\e} \ell^{\e}M^{1+\e} \frac{\log q}{a(q)}.
\end{align*}
%
%
It follows that when $\alpha$ is major arc, we get that
$$\sum_{1 \leq |n| \leq q} \frac{\chi(n)\psi(n)}{n}e(n\alpha) \ll_{\e} \ell^{\e}M^{1+\e} \frac{\log q}{a(q)} + \log_2 q.$$
Combining this with the minor arc case yields
$$\max_{\alpha \in [0,1]} \max_{1 \leq N \leq q} \left|\sum_{1 \leq n \leq N} \frac{\chi(n)\psi(n)}{n}e(n\alpha)\right| \ll_{\e} (\log q)\left(\frac{(\log M)^{5/2}}{\sqrt{M}} + \ell^{\e}M^{1+\e}a(q)^{-1}\right) + \log_2 q.$$
Choosing $M = a(q)^{2/3}$, the upper bound above becomes
\begin{align*}
\max_{\alpha \in [0,1]} \max_{1 \leq N \leq q} \left|\sum_{1 \leq n \leq N} \frac{\chi(n)\psi(n)}{n}e(n\alpha)\right| &\ll_{\e} \ell^{\e}(\log q)a(q)^{-1/3+\e}.
\end{align*}
This completes the proof.
\end{proof}
\begin{proof}[Proof of Proposition \ref{KeyProp2}]
Let $\ell = 5$ and let $\psi = \left(\frac{\cdot}{\ell}\right)$. Then $\psi$ is an even, primitive character with prime conductor. Combining Corollary \ref{BETTERLOGMEANFORCHI} with Lemma \ref{EXPSUMBD}, we get that
\begin{align*}
\max_{1 \leq N \leq q} \left|\sum_{1 \leq n \leq N} \frac{\chi(n)}{n}\right| &\ll \sqrt{\ell} \max_{\alpha \in [0,1]} \max_{1 \leq N \leq q} \left|\sum_{1 \leq |n| \leq N} \frac{\chi(n)\psi(n)}{n}e(n\alpha)\right| \ll_{\e} (\log q)a(q)^{-1/3+\e}.
\end{align*}
This completes the proof.
\end{proof}
\begin{proof}[Proof of Theorem \ref{THMGEN}]
The result is immediate (by the P\'{o}lya-Vinogradov inequality) if $t > q$, so in what follows, we shall assume that $t \leq q$. \\
In keeping with previous notation, set
$$\xi(t) := \left(\frac{\log a(t)}{13\log_2 a(t)}\right)^{\frac{1}{19g^2}},$$
and $f(t) := a(t)^{1/6}$. Note that this is non-decreasing since $a$ is. By Lemma \ref{Reduce}, the claim of the proposition is trivial if $\sum_{p|q} 1/p \geq \log \xi(q)$. Thus, we suppose henceforth that $q \in \mc{Q}(\xi)$. \\
Assume that 
$$\left|\frac{1}{t}\sum_{n \leq t} \chi(n)\right| > 1/\xi(t),$$
for some $t \in (q^{1/f(q)},q]$, noting that $q^{1/f(q)} > \exp(\sqrt{\log q})$ since $a(t) \leq \log t$ for all large $t$. By Proposition \ref{CESTOLOG} a) (with $k = g$), it follows that
$$\max_{\sqrt{t}< x \leq t} \left|\frac{1}{\log x}\sum_{n \leq x} \frac{\chi(n)}{n} \right| \gg \xi(t)^{-38g^2\xi(t)^{19g^2}} \geq \xi(q)^{-38g^2\xi(q)^{19g^2}}.$$
Let $N$ be the point in $(\sqrt{t},t]$ maximizing the LHS in this last expression. By Proposition \ref{KeyProp2}, the above is
\begin{align*}
\frac{1}{\log N}\left|\sum_{1 \leq n \leq N} \frac{\chi(n)}{n}\right| \ll \frac{f(q)}{\log q} \max_{1 \leq N' \leq q} \left|\sum_{1 \leq n \leq N'} \frac{\chi(n)}{n}\right| \ll_{\e} f(q)a(q)^{-1/3+\e}.
\end{align*}
Combining this with the above, rearranging and taking logarithms, we get
\begin{align*}
(1/6-\e) \log a(q) = (1/3-\e)\log a(q)-\log f(q) \leq 38g^2\xi(q)^{19g^2} \log \xi(q) + O_{\e}(1).
\end{align*}
On the other hand, as $\xi(q)^{19g^2} = \frac{\log a(q)}{13\log\log a(q)}$, the upper bound here is at most $(2/13) \log a(q)$ for large $q$, a contradiction for $\e$ sufficiently small. \\ 
This contradiction implies that for all $t > q^{1/f(q)}$, we get
$$\left|\sum_{n \leq t} \chi(n)\right| \leq t/\xi(t),$$
which implies the claim for all $q\in \mc{Q}(\xi)$ as well. The theorem is thus proved.
\end{proof}
\begin{rem}\label{RemGSBurg}
Note that if $\chi$ is assumed to be of odd order $g \geq 3$ then \eqref{BETTERPV} holds with $a(t) = (\log t)^{\delta_g}$ (see Remark \ref{RemOddOrd} above). 
The above proof can be modified to show cancellation in short sums of characters of the form $\chi(n)\psi(n)$, where $\psi$ is an odd character of small conductor. To do this, we combine Lemma \ref{LOGMEAN} with Proposition \ref{CESTOLOG} instead, which shows that if $t > \exp((\log t)^{1-\delta_g/2})$ and the conductor $\ell$ of $\psi$ is bounded then
\begin{align*}
\sum_{n \leq t} \chi(n)\psi(n) \ll_g t\left(\frac{\log_3 t}{\log_2 t}\right)^{\frac{1}{19(g\ell)^2}}.
\end{align*}
In fact, as $g$ is odd we can use part b) of Proposition \ref{CESTOLOG} instead of a) to replace the above bound by one of the shape $O_g\left(t(\log t)^{-c/(g\ell)^2}\right)$, with $c > 0$ sufficiently small in this same range of $t$. However, as mentioned in the introduction, cancellation in character sums of this kind can be shown to follow from Corollary 1.7 of \cite{GSBurg}. For instance, assume that $\psi$ is a quadratic character. As $g$ is odd we have $(\chi \psi)^g = \psi 1_{(n,q) = 1}$, which is a character of very small conductor relative to $q$, and thus necessarily witnessing cancellation along intervals $[1,t]$ at scales $t > q^{\e}$. Corollary 1.7 of \cite{GSBurg} then shows that this cancellation implies cancellation in the partial sums of $\chi \left(\frac{\cdot}{\ell}\right)$ itself.
\end{rem}
\section{A Converse of Theorem \ref{THMGEN}} \label{ConvSec}
In this section, we show that a converse of Theorem \ref{THMGEN}, namely Proposition \ref{Conv} is true. That is, we show that cancellation in short sums of primitive characters of a fixed order imply improvements in the P\'{o}lya-Vinogradov inequality for characters of that same fixed order. \\
To do this, we need two lemmata. The first is a very slight generalization of a result
of Granville-Soundarajan \cite{GSBurg}.
\begin{lem} \label{GS}
Let $t \in \mb{R}$ and let $x$ be large. Let $f: \mb{N} \ra \mb{U}$ be multiplicative. Put $\lambda := \mb{D}(f,n^{it};x)^2 + \log(1+|t|) + c$ for some large enough absolute constant, and $\eta = (\lambda e^{\lambda})^{-1}$. Then there is a $x^{\eta} \leq y \leq x$ such that
$$\left|\frac{1}{y}\sum_{n \leq y} f(n)\right| \gg \frac{e^{-\mb{D}(f,n^{it};x)^2}}{|1+it|}.$$
\end{lem}
\begin{proof}
The result in \cite{GSBurg} was stated for $t$ the minimizer of the distance $\mb{D}(f,n\mapsto n^{it};x)$ in the interval $[-\log x,\log x]$. However, the proof follows identically for any choice of $t \in \mb{R}$. 
\end{proof}
To glean a bit more information about characters that correlate with one another, we establish the following simple result.
\begin{lem}\label{Orders}
Let $\chi$ be a primitive character of order $g$ to a large modulus $q$, and suppose $\psi$ is a primitive character of order $k$ modulo $m$, with $m \leq \log q$. Then for any $0 \leq T \leq \log q$ at most one of the following holds: \\
i) $\mc{D}_{\chi\bar{\psi}}(q; T) \leq \frac{1}{3(gk)^2}\log\log q$, or else\\
ii) $k\nmid g$.\\
Moreover, i) can hold for at most one primitive character $\psi$ with conductor of size $\leq \log q$. 
\end{lem}
In the proof to follow, given $x,r \geq 2$ and 1-bounded multiplicative functions $f_1,f_2: \mb{N} \ra \mb{U}$, we define the distance function
\begin{align*}
\mb{D}_r(f_1,f_2;x) := \left(\sum_{p \leq x \atop p \nmid r} \frac{1-\text{Re}(f_1(p)\bar{f_2}(p))}{p}\right)^{\frac{1}{2}}.
\end{align*}
\begin{proof}
Assume instead that both i) and ii) hold. Let $t_0 \in [-T,T]$ minimize the distance here. Set $k^{\ast} := k/(k,g)$ and $g^{\ast} := g/(k,g)$, noting that $k^{\ast} > 1$ by assumption. As $k^{\ast}$ and $g^{\ast}$ are coprime, we can choose $r,s \in \mb{Z}$ with $0 < s < k^{\ast}$ such that $rk^{\ast} + sg^{\ast} = 1$. Now, set $t_0' := sgt_0$. We then have
\begin{align*}
\mb{D}(\psi^{(k,g)},n^{it_0'};q)^2 &= \mb{D}_{qm}(\psi^{(k,g)(rk^{\ast} + sg^{\ast})},n^{it_0'};q)^2 +O\left(\sum_{p|qm} \frac{1}{p}\right) \\
&= \mb{D}_{qm}(\psi^{sg},n^{it_0'};q)^2 +O\left(\log\log\log q\right)\\
&= \mb{D}_{qm}((\psi \bar{\chi})^{sg},n^{it_0'};q)^2 + O(\log\log\log q).
\end{align*}
By the triangle inequality, the bound $s < k$ and i), we have
\begin{align*}
\mb{D}_{qm}((\psi\bar{\chi})^{sg},n^{it_0'};q) &= \mb{D}(\psi^{sg},(\chi n^{it_0})^{sg};q) +O(\log\log\log q)\\
&\leq sg\mb{D}(\psi,\chi n^{it_0};q) +O(\log\log\log q)= s g\mc{D}_{\chi\bar{\psi}}(q;T)^{\frac{1}{2}} +O(\log\log\log q)\\
&\leq \left(\frac{s}{\sqrt{3}k} + o(1)\right)\sqrt{\log\log q} < \left(\frac{1}{\sqrt{3}}+o(1)\right)\sqrt{\log\log q}.
\end{align*}
On the other hand, as $(k,g) < k$, $\psi^{(k,g)}$ is a non-principal character modulo $m$ so that the effective version of Siegel's theorem (see, e.g., Theorem 8.21 of \cite{Ten2}) gives
\begin{align*}
\mb{D}(\psi^{(k,g)},n^{it_0'};q)^2 &= \log\log q - \log|L(1+1/\log x + it_0',\psi^{(k,g)})| + O(1) \\
&\geq \log\log q - \frac{1}{2}\log m - O(\log\log(kgm(T+2))+\sum_{p|m} 1/p) \\
&\geq \frac{1}{2}\log \log q - O(\log\log\log q),
\end{align*}
This yields the contradiction
\begin{align*}
\left(\frac{1}{2}-o(1)\right)\log\log q \leq \mb{D}(\psi^{(k,g)}, n^{it_0'};x)^2 \leq \left(\frac{1}{3}+o(1)\right)\log\log q,
\end{align*}
and proves the first claim. \\
For the uniqueness of $\psi$ we apply Lemma 3.1 of \cite{BaGrSo} (with $j = 2$), which shows that for any two characters $\psi_1,\psi_2$ to moduli of size $O(\log q)$, we have
\begin{equation*}
\max\{\mc{D}_{\chi\bar{\psi}_1}(q;T),\mc{D}_{\chi\bar{\psi}_2}(q;T)\} \geq \left(1-\frac{1}{\sqrt{2}}-o(1)\right)\log\log q \geq \frac{1}{4}\log\log q,
\end{equation*}
for $q$ large enough. Thus, since $g \geq 2$, at most one character $\psi$ of conductor $\leq \log q$ is such that 
$$\mc{D}_{\chi\bar{\psi}}(q;T) \leq \frac{1}{3(kg)^2}\log\log q \leq \frac{1}{12}\log\log q,$$
as required.
\end{proof}
\begin{proof}[Proof of Proposition \ref{Conv}]
We begin by considering the first claim of the proposition, and assume the contrary. Then we can find a primitive character $\chi$ such that 
$$\max_{1 \leq t \leq q} \left|\sum_{1 \leq n \leq t} \chi(n)\right| \geq \delta\sqrt{q}\log q,$$
where we have put $\delta := \e\log(1/\e)$. By the P\'{o}lya Fourier expansion (see Lemma \ref{CASEWITHPAR}) and primitivity, it follows that
$$\sup_{\alpha \in [0,1]} \left|\sum_{1 \leq |n| \leq q} \frac{\chi(n)(1-e(n\alpha))}{n}\right| \geq \frac{\delta}{2}\log q.$$
By the triangle inequality, 
$$\sup_{\beta \in [0,1]} \left|\sum_{1 \leq |n| \leq q} \frac{\chi(n)e(n\beta)}{n}\right| \geq \frac{\delta}{4} \log q.$$
Conversely, equation (7.1) in \cite{LamMan} (with $y = Q = q$) shows that
$$\sup_{\beta \in [0,1]} \left|\sum_{1 \leq |n| \leq q} \frac{\chi(n)e(n\beta)}{n} \right| \ll \frac{\sqrt{m}}{\phi(m)} (\log q)\exp\left(-\mc{D}_{\chi\bar{\psi}}(q; (\log q)^{-7/11})\right) + (\log q)^{9/11},$$
where $\psi$ is a primitive character to a conductor $m \leq \log q$ that is of opposite parity from $\chi$ and that minimizes the distance $\mc{D}_{\chi\bar{\psi}}(q;(\log q)^{-7/11})$. 
As $\delta \gg (\log q)^{-1/11}$, it follows that there must be a primitive character $\psi$ modulo $m$ with $m \ll \delta^{-2}$ of opposite parity from $\chi$ such that
$$\frac{\sqrt{m}}{\phi(m)} (\log q) e^{-\mc{D}_{\chi\bar{\psi}}(q;(\log q)^{-7/11})} \gg \delta \log q,$$
whence we get
$$\mc{D}_{\chi\bar{\psi}}(q;(\log q)^{-7/11})  \leq \log(1/\delta) + O(1) \leq \frac{1}{3(mg)^2}\log\log q + O(1),$$
in light of our assumed lower bound $\delta \geq Cg\left(\frac{\log_3 q}{\log_2 q}\right)^{1/2}$, with $C>0$ sufficiently large. Lemma \ref{Orders} then implies that $\text{ord}(\psi)|\text{ord}(\chi)$. \\
Let $t_0$ be the choice of $t \in [-(\log q)^{-7/11},(\log q)^{-7/11}]$ that minimizes $t\mapsto \mb{D}(\chi\bar{\psi},n^{it};q)^2$. By Lemma \ref{GS}, we can set $\lambda = \mb{D}(\chi\bar{\psi},n^{it_0};q)^2 + c$ for some large $c > 0$ to get that for some $y \in (q^{(\lambda e^{\lambda})^{-1}},q]$,
\begin{align*}
\frac{1}{y} \sum_{n \leq y} \chi(n)\bar{\psi}(n)\gg e^{-\lambda} \gg e^{-\mc{D}_{\chi\bar{\psi}}(q;(\log q)^{-7/11})} \gg_{\delta} 1.
\end{align*}
As $\lambda \leq \log(1/\delta) + O(1)$, we have $y > q^{\e}$ with $\e > \lambda^{-1}e^{-\lambda} \gg \delta/\log(1/\delta)$. \\
To complete the proof of the first claim, it remains to show that $\psi$ can be selected independently of $\delta$.
Indeed, suppose $\psi_{1}$ and $\psi_{2}$ are primitive characters to respective conductors $m_1,m_2$ that satisfy the conclusion of the first claim for parameters $\delta_1, \delta_2 > (\log_2 q)^{-1/2}$. As mentioned above, we have $m_j \ll \delta_j^{-2}$ for $j = 1,2$. Let $t_1$ and $t_2$ be the corresponding scales at which the partial sums of $\chi\bar{\psi}_{1}$ and $\chi\bar{\psi}_{2}$ are large in the above sense, and let $y := \max\{t_1,t_2\}$. By \eqref{HMT} we have that 
$$\mc{D}_{\chi\bar{\psi}_j}(y;(\log y)^2) \ll \log(1/\delta_j),$$
for $j = 1,2$. The triangle inequality for $\mb{D}$ thus gives that for some $u \in [-2(\log y)^2,2(\log y)^2]$, 
$$\mb{D}(\psi_1,\chi_0\psi_2n^{iu};y)^2 \ll \log(1/\delta_1) + \log(1/\delta_2) = O(\log_3 q),$$
where $\chi_0$ is the principal character modulo $q$. On the other hand, as in the proof of Lemma \ref{Orders}, if $\psi_1 \neq \psi_2$ then we have
\begin{align*}
\mb{D}(\psi_1,\chi_0\psi_2 n^{iu};y)^2 &= \log\log y - \log|L(1+1/\log y + iu,\psi_1\bar{\psi}_2)| + O(\sum_{p|q} 1/p) \\
&\geq \log\log q - \frac{1}{2}\log(m_1m_2)- O(\log_3 q) \\
&\geq (1-o(1))\log\log q - \log(1/(\delta_1\delta_2)).
\end{align*}
This is a contradiction, given $\delta_1,\delta_2 > (\log_2 q)^{-1/2}$. Hence, $\psi_1 = \psi_2$, and we can take $\psi$ independent of $\delta$ in this range.\\
We turn to the second claim of the proposition. Assume that the conclusion is false for an odd primitive character $\chi$ modulo $q$ of order $g$. Thus, we can find a $\delta > 0$ such that
$$\max_{1 \leq t \leq q} \left|\sum_{n \leq t} \chi(n)\right| \geq \delta \sqrt{q}\log q.$$
The first claim of the proposition implies that there is an even primitive character $\psi$ of conductor $m \ll 1$ and a $t > q^{\e}$ (with $\e \gg \delta/\log(1/\delta)$) such that 
\begin{equation}\label{ContraBd}
\sum_{n \leq t} \chi(n)\bar{\psi}(n) \gg_{\delta} t.
\end{equation}
Put $\tilde{q} := [q,m]$. Now, write $\chi \bar{\psi} = \xi^{\ast}\xi_0$, where $\xi^{\ast}$ is primitive modulo $\tilde{q}^{\ast}$ and $\xi_0$ is a principal character modulo $\tilde{q}_0$, with $\tilde{q} = \tilde{q}^{\ast}\tilde{q}_0$. We note that since $\chi$ and $\psi$ are both primitive, $\tilde{q}_0 \leq \min\{q,m\} = m \ll 1$, and moreover since $\text{ord}(\psi)|g$, $\text{ord}(\xi^{\ast}) = \text{ord}(\chi \bar{\psi})$ must divide $g$. \\
Clearly, $\xi^{\ast}$ is an odd character since $\chi\bar{\psi}$ is. By hypothesis, we thus know that, uniformly in $x > q^{\e} \geq (\tilde{q}^{\ast})^{\e}$,
\begin{align} \label{CancPrim}
\sum_{n \leq x} \xi^{\ast}(n) = o(x).
\end{align}
We apply this to the estimation of partial sums of $\chi\bar{\psi}$. Indeed,
\begin{align*}
\left|\sum_{n \leq t} \chi(n)\bar{\psi}(n)\right| = \left|\sum_{d|\tilde{q}_0} \mu(d)\chi(d)\bar{\psi}(d)\sum_{n \leq t/d} \xi^{\ast}(n)\right| \leq \sum_{d|\tilde{q}_0 \atop d < t/(\tilde{q}^{\ast})^{\e/2}} \left|\sum_{n \leq t/d} \xi^{\ast}(n)\right| + \sum_{d|\tilde{q}_0 \atop d \geq t/(\tilde{q}^{\ast})^{\e/2}} \left|\sum_{n \leq t/d}\xi^{\ast}(n)\right|.
\end{align*}
To the first double sum we apply \eqref{CancPrim}, giving the contribution
$$\sum_{d|\tilde{q}_0 \atop d < t/(\tilde{q}^{\ast})^{\e/2}} \left|\sum_{n \leq t/d} \xi^{\ast}(n)\right| = o\left(t\sum_{d|\tilde{q}_0}1/d\right) = o(t),$$
in light of the bound $\tilde{q}_0 \ll 1$. Estimating the second sum trivially, we get
$$\ll t\sum_{d|q_0 \atop d \geq t/(\tilde{q}^{\ast})^{\e/2}} 1/d \ll (\tilde{q}^{\ast})^{\e/2} \tau(q_0) \ll_{\e} t q^{-\e/2}.$$
Thus, on the whole, we find that
$$\sum_{n \leq t} \chi(n)\bar{\psi}(n) = o(t),$$
at the scale $t$. This contradicts \eqref{ContraBd} for $\delta >0$ given, so the second claim of the proposition also follows.
%
%
\end{proof}
\begin{rem} \label{RemFrGo}
Since the first claim of the proposition applies to even characters $\chi$ of fixed order as well, the proof of the second claim above can also be used to show that cancellation in short sums for all odd characters of order dividing $g$ implies cancellation in short sums for all even characters of order $g$. Together with our Theorem \ref{THMGEN}, this shows that if the P\'{o}lya-Vinogradov inequality is improved for all odd primitive characters of orders dividing $g$ then all even primitive characters of order $g$ exhibit cancellation in their short sums. This complements the work of \cite{FrGo}, which showed that in the case $g = 2$ improvements in P\'{o}lya-Vinogradov for all even primitive quadratic characters implies cancellation in short sums of all odd primitive quadratic characters. 
\end{rem}
\section{Obstructions in Extending Theorem \ref{THMGEN} to all Odd Primitive Characters}
In this section, we explain the shortcoming in our argument that prevents us from extending Theorem \ref{THMGEN} to odd characters, irrespective of their order. \\
The key simplifying feature of the bounded order case is the fact that if a character $\chi$ has large C\'{e}saro partial sums, $\chi$ must correlate heavily with 1, i.e., that $\mb{D}(\chi,1;x)^2$ is small. In general, this is not a condition furnished by Hal\'{a}sz' theorem, failing when $k$ is exceedingly large (e.g., on the order of $\phi(q)$, where $q$ is the modulus of $\chi$). \\
Nevertheless, Hal\'{a}sz' theorem does provide a $t_0 \in [-T,T]$ such that $\mb{D}(\chi,n^{it_0};x)$ (for suitably chosen $T$), and one might try to exploit the argument in Section 2 with $h := \chi n^{-it_0}$ instead of $\chi$ itself. The following result of Granville and Soundararajan suggests that this strategy might be effective.
\begin{lem}[Theorem 12.1 in \cite{GrSoBOOK}]
Let $f : \mb{N} \ra \mb{U}$ be multiplicative. Let $t_0 \in [-(\log x)^2,(\log x)^2]$ be such that $\mb{D}(f,n^{it_0};x)^2$ is minimal. Then
$$\frac{1}{x}\sum_{n \leq x} f(n) = \frac{x^{it_0}}{1+it_0} \cdot \frac{1}{x}\sum_{n \leq x} f(n)n^{-it_0} + O\left(\frac{\log_2 x}{(\log x)^{2-\sqrt{3}}}\right).$$
\end{lem}
Now, suppose we wanted to prove a version of Proposition \ref{CESTOLOG} for arbitrary multiplicative functions taking values in $S^1 \cup \{0\}$. Assume for simplicity that $\xi(x) = o\left(\frac{(\log x)^{2-\sqrt{3}}}{\log_2 x}\right).$ Then the lemma suggests that a hypothesis like $\frac{1}{x}\sum_{n \leq x} f(n) \gg 1/\xi(x)$ implies, in fact, that $$\left|\frac{1}{x(1+it_0)}\sum_{n \leq x}\tilde{f}(n)\right| > 1/(2\xi(x)),$$
where $\tilde{f}(n) := f(n)n^{-it_0}$. The crucial point here is that $\mb{D}(\tilde{f},1;x)^2$ is now small, and this was of importance in the proof of Proposition \ref{CESTOLOG} (in that the lower bound arising from Hildebrand's theorem for the C\'{e}saro partial sums of $1\ast 1\ast \tilde{f}\ast\bar{\tilde{f}}$ depends on $\mb{D}(\tilde{f},1;x)^2$ precisely). Running through the machinery of Lemma \ref{PRETDISTEST} and Proposition \ref{CESTOLOG}, we may conclude that 
$$\max_{\sqrt{x} < y \leq x} \left|\frac{1}{\log y} \sum_{n \leq y} \frac{\tilde{f}(n)}{n}\right|$$
is also large in a manner depending on $\xi(x)$ and $|t_0|$. \\
Now in order to proceed as in Section 3, where we selected $f= \chi\psi$, we must find a way of replacing the logarithmic mean of $\tilde{f}$ by that of $f$ in some way, since Lemma \ref{CASEWITHPAR} relies heavily on the periodicity of $f$, and $\tilde{f}$ is not periodic. However, unlike in the case of C\'{e}saro partial sums, we do not expect that the logarithmically averaged partial sums of $\tilde{f}$ and $f$ are of the same size. Indeed, if we take a simple example like $f \equiv 1$, the logarithmic partial sums of $\tilde{f}$ behaves like $\zeta(1+it_0)$, which is of size $1/|t_0|$ for small $|t_0|$ and of size $O(\log(2+|t_0|))$ for larger $t_0$. If $|t_0|\log y$ is large enough, in particular, then this is small compared to $\log y$. \\
It would be interesting to determine how an appeal to the function $\tilde{f}$ might be circumvented here.
\section*{Acknowledgments}
\noindent We would like to thank Andrew Granville, Oleksiy Klurman and Maksym Radziwi\l\l \ for their helpful advice and encouragement. 

\end{document}